\documentclass[12pt]{amsart}
\usepackage{}

\usepackage{amsmath}
\usepackage{amsfonts}
\usepackage{amssymb}
\usepackage[all]{xy}           

\usepackage{bbding}
\usepackage{txfonts}
\usepackage{amscd}

\usepackage[shortlabels]{enumitem}
\usepackage{ifpdf}
\ifpdf
  \usepackage[colorlinks,final,backref=page,hyperindex]{hyperref}
\else
  \usepackage[colorlinks,final,backref=page,hyperindex,hypertex]{hyperref}
\fi
\usepackage{tikz}
\usepackage[active]{srcltx}

\makeatletter

\newtheorem{thm}{Theorem}[section]
\newtheorem{lem}[thm]{Lemma}
\newtheorem{cor}[thm]{Corollary}
\newtheorem{pro}[thm]{Proposition}
\newtheorem{ex}[thm]{Example}
\newtheorem{rmk}[thm]{Remark}
\newtheorem{defi}[thm]{Definition}

\newcommand {\emptycomment}[1]{}

\newcommand{\Lei}{\mathsf{Leib}}
\newcommand{\nc}{\newcommand}
\newcommand{\delete}[1]{}

\nc{\CV}{\mathbf{C}}

\newcommand{\lon }{\,\rightarrow\,}
\newcommand{\be }{\begin{equation}}
\newcommand{\ee }{\end{equation}}

\newcommand{\g}{\mathfrak g}
\newcommand{\h}{\mathfrak h}
\newcommand{\la}{\mathfrak G} 

\newcommand{\huaB}{\mathcal{B}}

\newcommand{\huaL}{\mathcal{L}}

\newcommand{\huaP}{\mathcal{P}}

\newcommand{\huaH}{\mathcal{H}}

\newcommand{\huaZ}{\mathcal{Z}}

\newcommand{\frkC}{\mathfrak C}

\newcommand{\frkT}{\mathfrak T}

\newcommand{\half}{\frac{1}{2}}

\newcommand{\Courant}[1]{\left\llbracket  #1\right\rrbracket }

\newcommand{\Id}{{\rm{Id}}}

\newcommand{\br}[1]{   [ \cdot,    \cdot  ]   }

\newcommand{\dM}{\mathrm{d}}

\newcommand{\Hom}{\mathrm{Hom}}

\newcommand{\Der}{\mathrm{Der}}
\newcommand{\Lie}{\mathrm{Lie}}

\newcommand{\Nij}{\mathrm{Nij}}

\newcommand{\gl}{\mathfrak {gl}}

\newcommand{\LL}{Leibniz-Lie~}

\nc{\oprn}{\theta}
\newcommand{\B}{\mathsf{B}}
\newcommand{\NR}{\mathsf{NR}}

\newcommand{\ad}{\mathrm{ad}}
\newcommand{\pr}{\mathrm{pr}}

\newcommand{\Img}{\mathrm{Im}}

\nc{\mlabel}[1]{\label{#1}}  
\nc{\mcite}[1]{\cite{#1}}  
\nc{\mref}[1]{\ref{#1}}  
\nc{\mbibitem}[1]{\bibitem{#1}} 
\newcommand{\C}{\mathbb{C}}
\begin{document}

\title[]{Nonabelian embedding tensors}

\author{Rong Tang}
\address{Department of Mathematics, Jilin University, Changchun 130012, Jilin, China}
\email{tangrong@jlu.edu.cn}

\author{Yunhe Sheng}
\address{Department of Mathematics, Jilin University, Changchun 130012, Jilin, China}
\email{shengyh@jlu.edu.cn}



\begin{abstract}
In this paper, first we introduce the notion of a nonabelian embedding tensor, which is a nonabelian generalization of an embedding tensor. Then we introduce the notion of a Leibniz-Lie algebra, which is the underlying algebraic structure of a nonabelian embedding tensor, and can also be viewed as a nonabelian generalization of a Leibniz algebra. Next using the derived bracket, we construct a differential graded Lie algebra, whose Maurer-Cartan elements are exactly nonabelian embedding tensors. Consequently, we obtain the differential graded Lie algebra that governs deformations of a nonabelian embedding tensor. Finally,  we define the  cohomology of a nonabelian embedding tensor and use the second cohomology group to characterize  linear deformations.

\end{abstract}


\keywords{cohomology, deformation,  nonabelian embedding tensor,   Leibniz-Lie algebra}

\maketitle

\tableofcontents

\allowdisplaybreaks


\section{Introduction}
The notion of embedding tensors \cite{NS} can be traced back to the study of the gauged supergravity theory. Embedding tensors were used to construct  N =8 supersymmetric gauge theories and   to study the Bagger-Lambert theory of multiple M2-branes in \cite{BdH}. Embedding tensors were also used to study $E_{11}$   \cite{BND} and deformations of maximal 9-dimensional supergravity  \cite{FOT}. Recently, this topic has attracted much attention of the mathematical physics world \cite{KS,Lavau-1}. An embedding tensor on a Lie algebra $\g$ with respect to a representation $\rho:\g\to\gl(V)$ is a
linear map $T:V\lon\g$ satisfying
$$
[Tu,Tv]=T(\rho(Tu)v),\quad\forall u,v\in V.
$$
An embedding tensors gives rise to a tensor hierarchy algebra  \cite{Pal}, while the existing constructions \cite{LavauS,Lavau-2}   are still mysterious.

Embedding tensors have been already studied in mathematics under the name of averaging operators since 1960s. Luxuriant studies on averaging operators were done for Banach algebras \cite{Kelley,Miller-1,Rota} from the analysis point of view. Recently, there are more and more studies   of averaging operators from the algebraic point of view. Averaging operators on various types of  algebras  were studied in \cite{Aguiar}. In particular, an averaging operator on a Lie algebra gives rise to a Leibniz algebra.  More generally,   averaging
operators  are studied on the algebras over the general operad $\huaP$ and give rise to the di-$\huaP$ algebras and tri-$\huaP$ algebras \cite{GK}.
 Averaging operators are related to the
procedure of replication, Hadamard products  and Manin white products in the operad theory \cite{GK,Pei-Bai-Guo-Ni,PG,Uchino-2}, and can be applied  to study  double
algebras, classical Yang-Baxter equations and Gr\"obner-Shirshov bases \cite{GZ,Goncharov,Kolesnikov,ZGR}.

Note that the above vector space $V$ in the definition of an embedding tensor is understood as the space that 1-form fields take values in. Since for the  classical gauge field theories, a 1-form field  usually takes values in a Lie algebra, so it is natural to replace the vector space $V$ by a Lie algebra $\h$. This will further need an action of the Lie algebra $\g$ on the Lie algebra $\h$, and lead to the `nonabelian embedding tensor', which is the  main subject that discussed in this paper.  Such generalization for averaging operators on associative algebra were already considered   in \cite{GK,Pei-Bai-Guo-Ni}, which can be understood as weighted averaging operators.

We introduce a new algebraic structure, which is called a Leibniz-Lie algebra. A Leibniz-Lie algebra contains a Lie algebra structure and a binary operation such that some compatibility conditions hold (Definition \ref{Leibniz-dendriform}). In particular, if the Lie algebra structure is trivial, then it reduces to a Leibniz algebra. So a Leibniz-Lie algebra can be viewed as a nonabelian generalization of a Leibniz  algebra. A nonabelian embedding tensor naturally induces a Leibniz-Lie algebra.  The concrete relations can be summarized by the following diagram:
\begin{equation*}
\begin{split}
\xymatrix{
	\text{Leibniz algebras} \ar^{\text{nonabelianzation}}[rrr] &&& \text{Leibniz-Lie algebras}\\
	\text{embedding tensors} \ar^{\text{nonabelianzation}}[rrr] \ar^{\text{descendent}}[u] &&& \text{nonabelian embedding tensors.} \ar^{\text{descendent}}[u]
}
\end{split}
\mlabel{eq:reln}
\end{equation*}

To justify nonabelian embedding tensors are indeed good mathematical structures, we further  construct a differential graded Lie algebra such that its Maurer-Cartan elements are exactly nonabelian embedding tensors. As a byproduct, we obtain the differential graded Lie algebra that governs  deformations of a nonabelian embedding tensor. The deformation of algebraic structures were originated from  the pioneer
work~\cite{Ge} of Gerstenhaber. In general, deformation theory was developed for algebras over binary quadratic operads by Balavoine~\cite{Bal}. The deformation theory of embedding tensors were studied in \cite{STZ}. In the associative algebra context,
the deformation theory of averaging operators   was studied in \cite{Das} via the derived brackets \cite{Kosmann-Schwarzbach-1,Uchino-2}, and the deformation theory of averaging
associative algebras was   studied in \cite{WZ}.

A nonabelian embedding tensor induces a Leibniz-Lie algebra, which further gives rise to a Leibniz algebra, which is called the descendent Leibniz algebra of the nonabelian embedding tensor. We introduce a  cohomology theory of  nonabelian embedding tensors via the Loday-Pirashvili cohomology of the descendent Leibniz algebras. As applications, we use the second   cohomology group   to characterize   equivalent linear deformations of a  nonabelian embedding tensor. Moreover, we introduce the notion of a Nijenhuis element for a nonabelian embedding tensor, and show that trivial linear deformations   are generated by   Nijenhuis elements.

The paper is organized as follows. In Section \ref{sec:L}, we introduce the notion of nonabelian embedding tensors, and show that nonabelian embedding tensors can be characterized by graphs of the nonabelian hemisemidirect product Leibniz algebra. In Section \ref{Leibniz-Lie}, we define a new algebraic structure, which is called a Leibniz-Lie algebra. We show that a nonabelian embedding tensor naturally induces a Leibniz-Lie algebra.   In Section \ref{Deformation-DGLA}, we construct a differential graded Lie algebra whose Maurer-Cartan elements are nonabelian embedding tensors.   In Section \ref{Cohomology}, we introduce a cohomology theory of nonabelian embedding tensors. In Section \ref{Linear-Deformation}, we study   linear deformations  of nonabelian embedding tensors using the established cohomology theory.

\newpage
\section{Nonabelian embedding tensors and Leibniz algebras}\label{sec:L}
In this section, we introduce the notion of nonabelian embedding tensors, which are  the nonabelian generalization of embedding tensors. We show that nonabelian embedding tensors can be characterized by graphs of the nonabelian hemisemidirect product Leibniz algebra.

A {\bf Leibniz algebra}  is a vector space $\la$ together with a bilinear operation $[-,-]_\la:\la\otimes\la\lon\la$ such that
\begin{eqnarray*}
\label{Leibniz}[x,[y,z]_\la]_\la=[[x,y]_\la,z]_\la+[y,[x,z]_\la]_\la,\quad\forall x,y,z\in\la.
\end{eqnarray*}

A {\bf representation} of a Leibniz algebra $(\la,[-,-]_{\la})$ is a triple $(V;\rho^L,\rho^R)$, where $V$ is a vector space, $\rho^L,\rho^R:\la\lon\gl(V)$ are linear maps such that  for all $x,y\in\la$,
\begin{eqnarray*}
\rho^L([x,y]_{\la})&=&[\rho^L(x),\rho^L(y)],\\
\rho^R([x,y]_{\la})&=&[\rho^L(x),\rho^R(y)],\\
\rho^R(y)\circ \rho^L(x)&=&-\rho^R(y)\circ \rho^R(x).
\end{eqnarray*}
Here $[-,-]:\wedge^2\gl(V)\lon\gl(V)$ is the commutator Lie bracket on $\gl(V)$.

Let $(\g,[-,-]_\g)$ and $(\h,[-,-]_\h)$ be Lie algebras. A Lie algebra homomorphism $\rho:\g\lon\Der(\h)$ is called an {\bf action} of $\g$ on $\h$. An action $\rho$ of $\g$ on $\h$ is called a {\bf coherent action} if it satisfies
\begin{equation}\label{eq:extracon}
  [\rho(x)u,v]_\h=0,\quad \forall x\in\g, u,v\in \h.
\end{equation}
We denote a coherent action by $(\h;\rho)$.  Denote by $\overline{\Der}(\h)\subset \Der(\h)$ the Lie subalgebra of the derivation Lie algebra $\Der(\h)$ satisfying
$$
\overline{\Der}(\h)=\{D\in\Der(\h)|[Du,v]_\h=0,~\forall u,v\in\h\}.
$$
The Lie algebra $\overline{\Der}(\h)$ is called the {\bf coherent derivation Lie algebra}. A coherent action can be alternatively described by a homomorphism from the Lie algebra $\g$ to  the  coherent derivation Lie algebra $\overline{\Der}(\h)$.
\begin{ex}{\rm
Recall that a  {\bf two-step nilpotent Lie algebra} \cite{Scheuneman} is a Lie algebra $\g$ such that
\begin{equation}\label{2-step nilpotent}
  [[x,y],z]=0,\,\,\forall x,y,z\in\g.
\end{equation}
For a two-step nilpotent Lie algebra, the adjoint action $\ad$ is a coherent action and  $\overline{\Der}(\g)=\Der(\g)$. 
}

\end{ex}

\begin{defi} \label{defi:O}
  A  {\bf nonabelian embedding tensor} on a Lie algebra $\g$ with respect to   a coherent action  $(\h;\rho)$ is a linear map $T:\h\to\g$ such that
 \begin{equation}\label{eq:nonET}
   [Tu,Tv]_\g=T\big(\rho(Tu)(v)+[u,v]_\h\big),\quad\forall u,v\in \h.
 \end{equation}
\end{defi}

    \begin{defi}
      Let $T$ and $T'$ be two nonabelian  embedding tensors on a Lie algebra $\g$ with respect to a coherent action  $(\h;\rho)$. A {\bf homomorphism} from $T'$ to $T$ consists of two Lie algebra homomorphisms  $\phi_\g:\g\to\g$ and  $\phi_\h:\h\to \h$ such that
      \begin{eqnarray}
        T\circ \phi_\h&=&\phi_\g\circ T',\mlabel{defi:isocon1}\\
        \phi_\h\rho(x)u&=&\rho(\phi_\g(x))\phi_\h(u),\quad\forall x\in\g, u\in \h.\mlabel{defi:isocon3}
      \end{eqnarray}
      In particular, if both $\phi_\g$ and $\phi_\h$ are  invertible,  $(\phi_\g,\phi_\h)$ is called an  {\bf isomorphism}  from $T'$ to $T$.
    \mlabel{defi:isoO}
    \end{defi}

\begin{rmk}\label{non-abelian}
When $(\h,[-,-]_\h)$ is an abelian Lie algebra, we obtain that $T$ is a usual embedding tensor \cite{KS,STZ}. In addition, when $\rho$ is the trivial  action of $\g$ on $\h$,   $T$ is a Lie algebra homomorphism from $\h$ to $\g$.
\end{rmk}

  Consider the direct sum Lie algebra $\overline{\Der}(\h)\oplus \h$, in which the Lie bracket is given by
  $$
  [A+u,B+v]=[A,B]+[u,v]_\h,\quad \forall A,B\in\overline{\Der}(\h),  u,v\in \h
  $$
 Define $\rho:\overline{\Der}(\h)\to \gl(\overline{\Der}(\h)\oplus \h)$ by
 $$
 \rho(A)(B+v)=Av,\quad \forall A,B\in \overline{\Der}(\h),~v\in \h.
 $$
Then it is straightforward to deduce that $\rho$ is a coherent action of the Lie algebra  $\overline{\Der}(\h) $ on the Lie algebra $\overline{\Der}(\h)\oplus \h$.

\begin{pro}\label{pro:pr}
 The projection $\pr:\overline{\Der}(\h)\oplus \h\to \overline{\Der}(\h)$ is a nonabelian embedding tensor on the Lie algebra  $\overline{\Der}(\h) $ with respect to the coherent action  $(\overline{\Der}(\h)\oplus \h;\rho).$
\end{pro}
\begin{proof}
  It follows from
  $$
  \pr\Big(\rho(\pr(A+u))(B+v)+[A+u,B+v]\Big)=[A,B]=[\pr(A+u),\pr(B+v)],
  $$
  for all $A,B\in \overline{\Der}(\h), u,v\in \h$.
\end{proof}

\begin{ex}\label{Heisenberg-Lie-ex}
{\rm
The {\bf Heisenberg Lie algebra} $H_3(\mathbb C)$ is the
three-dimensional  complex  Lie algebra   with the Lie bracket given by
\begin{eqnarray*}
[e_1,e_2]=e_3,\quad [e_1,e_3]=0,\quad  [e_2,e_3]=0,
\end{eqnarray*}
with respect to a basis $\{ e_1, e_2, e_3 \}$. It is obvious that the Heisenberg Lie algebra  $H_3(\mathbb C)$ is a two-step nilpotent Lie algebra.
A linear transformation
$T=\left(\begin{array}{ccc}r_{11}&r_{12}&r_{13}\\
r_{21}&r_{22}&r_{23}\\
r_{31}&r_{32}&r_{33}\end{array}\right)$ is a nonabelian embedding tensor with respect to the adjoint action if and only if
$
   [Te_i,Te_j]=T([Te_i,e_j]+[e_i,e_j]),~~ i,j\in\{1,2,3\}.
$
First by
\begin{eqnarray*}
 0&=&[Te_1,Te_1]= T([Te_1,e_1]+[e_1,e_1])\\
 &=&T[r_{11}e_1+r_{21}e_2+r_{31}e_3, e_1]=-r_{21}r_{13}e_1-r_{21}r_{23}e_2-r_{21}r_{33}e_3,
\end{eqnarray*}
we obtain
$
r_{13}r_{21}=0,~ r_{21}r_{23}=0,~ r_{21}r_{33}=0.
$
Similarly, we have
$$\begin{array}{rcl  rcl rcl}
  r_{12}r_{13}&=&0, & r_{12}r_{23}&=&0, &  r_{12}r_{33}&=&0,\\
(r_{11}+1)r_{13}&=&0, & (r_{11}+1)r_{23}&=&0, & r_{11}r_{22}-r_{12}r_{21}&=&(r_{11}+1)r_{33},\\
(r_{22}+1)r_{13}&=&0, & (r_{22}+1)r_{23}&=&0, & r_{12}r_{21}-r_{11}r_{22}&=&-(r_{22}+1)r_{33},\\
r_{11}r_{23}&=&r_{13}r_{21}, &r_{13}r_{23}&=&r_{23}^2=0, &r_{13}r_{21}-r_{11}r_{23}&=&-r_{23}r_{33},\\
r_{12}r_{23}&=&r_{13}r_{22}, &r_{13}r_{23}&=&r_{13}^2=0, &r_{13}r_{22}-r_{12}r_{23}&=&r_{13}r_{33}.
\end{array}
$$
Therefore, we have
   \begin{itemize}
     \item[\rm(i)] If $r_{13}=r_{23}=r_{33}=0$,  then   $T=\left(\begin{array}{ccc}r_{11}&r_{12}&0\\
r_{21}&r_{22}&0\\
r_{31}&r_{32}&0\end{array}\right)$ is a nonabelian embedding tensor on $H_3(\mathbb C)$ if and only if $r_{11}r_{22}=r_{12}r_{21}$. Moreover, for any $a,b,c,d,k\in\mathbb C$, we deduce that $T=\left(\begin{array}{ccc}c&d&0\\
kc&kd&0\\
a&b&0\end{array}\right)$ and $T=\left(\begin{array}{ccc}0&0&0\\
c&d&0\\
a&b&0\end{array}\right)$ are nonabelian embedding tensors on $H_3(\mathbb C)$.
     \item[\rm(ii)] If $r_{13}=r_{23}=0$ and $r_{33}\not=0$, then $r_{12}=r_{21}=0$,~$r_{11}=r_{22}$ and    $T=\left(\begin{array}{ccc}r_{11}&0&0\\
0&r_{11}&0\\
r_{31}&r_{32}&r_{33}\end{array}\right)$ is a nonabelian embedding tensor on $H_3(\mathbb C)$ if and only if   $r_{11}^2-r_{33}r_{11}-r_{33}=0$. Moreover, for any $a,b\in\mathbb C$ and a nonzero complex number $0\not=t\in\mathbb C$, we deduce that
$T=\left(\begin{array}{ccc}\frac{t+\sqrt{t^2+4t}}{2}&0&0\\
0&\frac{t+\sqrt{t^2+4t}}{2}&0\\
a&b&t\end{array}\right)$ and $T=\left(\begin{array}{ccc}\frac{t-\sqrt{t^2+4t}}{2}&0&0\\
0&\frac{t-\sqrt{t^2+4t}}{2}&0\\
a&b&t\end{array}\right)$ are  nonabelian embedding tensors on $H_3(\mathbb C)$.
   \end{itemize}
These are all the possibilities of nonabelian embedding tensors on
$H_3(\mathbb C)$ with respect to the adjoint representation.
}
\end{ex}

Let $(\g,[-,-]_\g)$ and $(\h,[-,-]_\h)$ be Lie algebras. Let $\rho:\g\lon\Der(\h)$ be a coherent action of $\g$ on $\h$. On the direct sum vector space $\g\oplus \h,$ define a bilinear bracket operation $[-,-]_\rho$ by
\begin{equation}
  [x+u,y+v]_\rho=[x,y]_\g+\rho(x)v+[u,v]_\h,\quad \forall x,y\in\g,~u,v\in \h.
\end{equation}

\begin{pro}\label{nonabelian hemisemidirect}
  With the above notations, $(\g\oplus \h,[-,-]_\rho)$ is a Leibniz algebra, which is called the {\bf nonabelian hemisemidirect product}, and denoted by $\g\ltimes^h_\rho \h$.
\end{pro}

\begin{proof}
  By the Jacobi identities for the Lie brackets $[-,-]_\g$ and $[-,-]_\h$, and the fact that $\rho$ is a homomorphism from the Lie algebra $(\g,[-,-]_\g)$ to the derivation Lie algebra $\Der(\h)$, we have
  \begin{eqnarray*}
    &&[x+u,[y+v,z+w]_\rho]_\rho-[[x+u,y+v]_\rho,z+w]_\rho-[y+v,[x+u,z+w]_\rho]_\rho\\
    &=&[x,[y,z]_\g]_\g-[[x,y]_\g,z]_\g-[y,[x,z]_\g]_\g+[u,[v,w]_\h]_\h-[[u,v]_\h,w]_\h-[v,[u,w]_\h]_\h\\
    &&+\rho(x)[v,w]_\h-[\rho(x)v,w]_\h-[v,\rho(x)w]_h+\rho(x)\rho(y)w-\rho([x,y]_\g)w-\rho(y)\rho(x)w\\
    &&+[u,\rho(y)w]_\h-\rho(y)[u,w]_\h\\
    &=&[u,\rho(y)w]_\h-\rho(y)[u,w]_\h.
  \end{eqnarray*}
  By \eqref{eq:extracon} and the fact that $\rho(y)\in\Der(\h)$, we deduce that $[u,\rho(y)w]_\h-\rho(y)[u,w]_\h=0$. Therefore, $(\g\oplus \h,[-,-]_\rho)$ is a Leibniz algebra.
\end{proof}


\begin{rmk}
In fact, the nonabelian hemisemidirect product $\g\ltimes^h_\rho \h$ is a split Leibniz algebra extension of the Lie algebra $\g$ by the Lie algebra $\h$ via the antisymmetric representation. See \cite[Proposition 1]{SMT} for more details
about nonabelian extensions of Leibniz algebras.
\end{rmk}

\begin{rmk}
Note that when $\h$ is abelian, the bracket $[-,-]_\rho$ is the {\bf hemisemidirect product} of a Lie algebra and its representation introduced in \cite{Kinyon} in the study of Courant algebroids. This kind of Leibniz algebras are very important, e.g. semisimple Leibniz algebras are of this form \cite{FW}.
\end{rmk}

Let $T:\h\to\g$ be a linear map. Denote by $G_T$ the graph of $T$, i.e. $G_T=\{Tu+u|\forall u\in \h\}$.

\begin{pro}
Let $(\g,[-,-]_\g)$ and $(\h,[-,-]_\h)$ be Lie algebras. Let $\rho:\g\lon\Der(\h)$ be a coherent action of $\g$ on $\h$. Then a linear map $T:\h\to\g$ is a nonabelian embedding tensor if and only if its graph $G_T$ is a Leibniz subalgebra of the Leibniz algebra $(\g\oplus \h,[-,-]_\rho)$.
\end{pro}
\begin{proof}
  We have
  \begin{eqnarray*}
    [Tu+u,Tv+v]_\rho=[Tu,Tv]_\g+\rho(Tu)v+[u,v]_\h,\quad \forall u,v\in\h,
  \end{eqnarray*}
  which implies that the graph $G_T$ is a Leibniz subalgebra if and only if \eqref{eq:nonET} holds, i.e. $T:\h\to\g$ is a nonabelian embedding tensor.
\end{proof}

Since $\h$ and $G_T$ are naturally isomorphic, we get the following conclusion immediately.

\begin{cor}\label{average-Leibniz}
 Let $T:\h\to\g$ be a nonabelian embedding tensor on a Lie algebra $\g$ with respect to a coherent action  $(\h;\rho)$. Then $(\h,[-,-]_T)$ is a Leibniz algebra, where the Leibniz bracket $[-,-]_T$ is given by
 \begin{equation}\label{new-leibniz}
   [u,v]_T=\rho(Tu)v+[u,v]_\h,\quad\forall u,v\in\h.
 \end{equation}
 Moreover, $T$ is a Leibniz algebra homomorphism from the Leibniz algebra $(\h,[-,-]_T)$ to the Lie algebra  $(\g,[-,-]_\g)$.
\end{cor}

The Leibniz algebra $(\h,[-,-]_T)$  is called the {\bf descendent Leibniz algebra} of the nonabelian embedding tensor $T$.

\begin{ex}\label{ex:prLeibniz}
{\rm
  Consider the nonabelian embedding tensor given in Proposition \ref{pro:pr}. Then by Corollary \ref{average-Leibniz}, one obtain a Leibniz algebra structure on $\overline{\Der}(\h)\oplus \h:$
  $$
  [A+u, B+v]_\pr=[A+u, B+v]+\rho(\pr(A+u))(B+v)=[A,B]+Av+[u,v]_\h,
  $$
for all $A,B\in \overline{\Der}(\h),~u,v\in \h$,  which is exactly the nonabelian hemisemidirect product of   $\overline{\Der}(\h)$ and $ \h$ given in Proposition \ref{nonabelian hemisemidirect}.
}
\end{ex}

The association of a Leibniz algebra from a nonabelian embedding tensor enjoys the functorial property.
\begin{pro}
Let $T$ and $T'$ be two nonabelian  embedding tensors on a Lie algebra $\g$ with respect to a coherent action  $(\h;\rho)$ and $(\phi_\g,\phi_\h)$   a homomorphism   from $T'$ to $T$. Then $\phi_\h$ is a homomorphism  of Leibniz algebras from $(\h,[-,-]_{T'})$ to $(\h,[-,-]_{T})$. 
\end{pro}

\begin{proof}
For all $u,v\in \h$, we
have
\begin{eqnarray*}
\phi_\h([u,v]_{T'})&=&\phi_\h(\rho(T'u)v+[u,v]_\h)
=\rho(\phi_\g(T'u))\phi_\h(v)+[\phi_\h(u),\phi_\h(v)]_\h\\
&=&\rho(T\phi_\h(u))\phi_\h(v)+[\phi_\h(u),\phi_\h(v)]_\h=[\phi_\h(u),\phi_\h(v)]_T.
\end{eqnarray*}
Then $\phi_\h$ is a homomorphism of Leibniz algebras from $(\h,[-,-]_{T'})$ to $(\h,[-,-]_{T})$.
\end{proof}

\section{Nonabelian embedding tensors and Leibniz-Lie algebras}\label{Leibniz-Lie}
In this section, we introduce the notion of a \LL algebra as the underlying algebraic structure of a  nonabelian embedding tensor. A \LL algebra gives rise to a Leibniz algebra, as well as nonabelian embedding tensors via different approaches.

\begin{defi}\label{Leibniz-dendriform}
A {\bf \LL algebra} $(\h,[-,-]_\h,\rhd)$ consists of a Lie algebra $(\h,[-,-]_\h)$ and a binary product $\rhd:\h\otimes\h\lon\h$ such that for all $x,y,z\in \h,$ one has
\begin{eqnarray}
\label{Post-1}x\rhd (y\rhd z)&=&(x\rhd y)\rhd z+y\rhd (x\rhd z)+[x,y]_\h\rhd z,\\
\label{Post-2}x\rhd[y,z]_\h&=&[x\rhd y,z]_\h=0.
\end{eqnarray}
A homomorphism from a   \LL algebra   $(\h,[-,-]_\h,\rhd)$  to  a   \LL algebra   $(\h',[-,-]_{\h'},\rhd')$ is a Leibniz algebra homomorphism $\phi_\h$ such that $\phi_\h(x\rhd y)=\phi_\h(x)\rhd'\phi_\h( y)$.
\end{defi}

\begin{rmk}
A Leibniz algebra $(\la,[-,-]_{\la})$ is naturally a \LL algebra, in which   the  Lie algebra structure is abelian.
\end{rmk}

\begin{ex}
{\rm
Let $H_3(\mathbb C)$ be the  Heisenberg Lie algebra. We define a binary product $\rhd:H_3(\mathbb C)\otimes H_3(\mathbb C)\lon H_3(\mathbb C)$ by
$$\begin{array}{rcl  rcl rcl rcl}
  e_1\rhd e_1&=&-e_3, & e_1\rhd e_2&=&e_3, &  e_2\rhd e_2&=&e_3,&  e_2\rhd e_1&=&-e_3\\
e_1\rhd e_3&=&0, & e_2\rhd e_3&=&0, & e_3\rhd e_3&=&0,&  e_3\rhd e_1&=&e_3\rhd e_2=0.
\end{array}
$$
Then $(H_3(\mathbb C),[-,-],\rhd)$ is a \LL algebra.
}
\end{ex}

\begin{thm}\label{Leibniz-up-Leibniz-dendriform}
Let $(\h,[-,-]_\h,\rhd)$ be a \LL algebra. Then the binary operation $[-,-]:\h\otimes \h\lon \h$ given by
\begin{eqnarray}
[x,y]=x\rhd y+[x,y]_\h,\,\,\,\,\forall x,y\in \h,
\end{eqnarray}
defines a Leibniz algebra structure on $\h$, which is denoted by  $\h^\Lei$ and called the {\bf subadjacent  Leibniz algebra}.

Moreover, define $\huaL:\h\lon\gl(\h)$ by
\begin{equation}\label{eq:L}
\huaL_xy=x\rhd y.
\end{equation} Then $(\h;\huaL,0)$ is a representation of the Leibniz algebra $(\h,[-,-])$ on $\h$.

\end{thm}
\begin{proof}
For all $x,y,z\in \h$, we have
\begin{eqnarray*}
[x,[y,z]=[x,y\rhd z+[y,z]_\h]&=&x\rhd (y\rhd z)+[x,y\rhd z]_\h+x\rhd [y,z]_\h+[x,[y,z]_\h]_\h\\
                             &\stackrel{\eqref{Post-2}}{=}&x\rhd (y\rhd z)+[x,[y,z]_\h]_\h.
\end{eqnarray*}
On the other hand, we have
\begin{eqnarray*}
[[x,y],z]+[y,[x,z]&=&[x\rhd y+[x,y]_\h,z]+[y,x\rhd z+[x,z]_\h]\\
                   &=&(x\rhd y)\rhd z+[x\rhd y,z]_\h+[x,y]_\h\rhd z+[[x,y]_\h, z]_\h\\
                   &&+y\rhd (x\rhd z)+[y,x\rhd z]_\h+y\rhd [x,z]_\h+[y,[x,z]_\g]_\h\\
                   &\stackrel{\eqref{Post-2}}{=}&(x\rhd y)\rhd z+[x,y]_\h\rhd z+[[x,y]_\h, z]_\h+y\rhd (x\rhd z)+[y,[x,z]_\h]_\h.
\end{eqnarray*}
By \eqref{Post-1}, we deduce that $(\h,[-,-])$ is a Leibniz algebra.

Also by   \eqref{Post-1},  we get
$
\huaL_{[x,y]}=[\huaL_x,\huaL_y],
$
which implies that  $(\h;\huaL,0)$ is a representation of the Leibniz algebra $(\h,[-,-])$ on $\h$.
\end{proof}

An embedding tensor naturally induces a Leibniz algebra. Now a nonabelian embedding tensor naturally induces a \LL algebra.

\begin{thm}\label{Leibniz-dendriform-under-o-operator}
  Let  $T:\h\to\g$ be a nonabelian embedding tensor on a Lie algebra $\g$ with respect to a coherent action  $(\h;\rho)$. Then there is a \LL algebra structure on the Lie algebra $(\h,[-,-]_\h)$ given by
\begin{eqnarray}
u \rhd v=\rho(Tu)v,\,\,\,\,\forall u,v\in \h.
\end{eqnarray}
\end{thm}

\begin{proof}
For all $u,v,w\in \h$, we have
\begin{eqnarray*}
u \rhd (v\rhd w)&=&\rho(Tu)\rho(Tv)w=\rho([Tu,Tv]_\g)w+\rho(Tv)\rho(Tu)w\\
                                  &\stackrel{\eqref{eq:nonET}}{=}&\rho\big(T(\rho(Tu)v+[u,v]_\h)\big)w+\rho(Tv)\rho(Tu)w\\
                                  &=&(u \rhd v)\rhd w+[u,v]_\h\rhd w+v \rhd (u\rhd w),
\end{eqnarray*}
which implies that \eqref{Post-1} in Definition \ref{Leibniz-dendriform} holds. Moreover, we have
\begin{eqnarray*}
u\rhd[v,w]_\h=\rho(Tu)[v,w]_\h=[\rho(Tu)v,w]_\h+[v,\rho(Tu)w]_\h.
\end{eqnarray*}
By \eqref{eq:extracon},   $u\rhd[v,w]_\h=[u\rhd v,w]_\h=0$. Thus,   $(\h,[-,-]_\h,\rhd)$ is a \LL algebra.
\end{proof}

\begin{rmk}
 The Leibniz algebra determined by the \LL algebra $(\h,[-,-]_\h,\rhd)$ according to  Theorem \ref{Leibniz-up-Leibniz-dendriform} is exactly the Leibniz algebra given in Corollary \ref{average-Leibniz}.
\end{rmk}

\begin{ex}{\rm
 Consider the nonabelian embedding tensor given in Proposition \ref{pro:pr}. Then by Theorem  \ref{Leibniz-dendriform-under-o-operator},  $(\overline{\Der}(\h)\oplus \h, [-,-], \rhd)$ is a \LL algebra, where $[-,-]$ is the direct sum Lie bracket and $\rhd$ is given by
  $$
   (A+u)\rhd( B+v)=\rho(\pr(A+u))( B+v)=Av,
  $$
for all $A,B\in \overline{\Der}(\h),~u,v\in \h$.

Furthermore, by Theorem \ref{Leibniz-up-Leibniz-dendriform}, the \LL algebra  $(\overline{\Der}(\h)\oplus \h, [-,-], \rhd)$ gives rise to a Leibniz algebra which is exactly the nonabelian hemisemidirect product Leibniz algebra  given in Example \ref{ex:prLeibniz}.
}
\end{ex}

It is straightforward to obtain the following result, which implies that the association given in Theorem \ref{Leibniz-dendriform-under-o-operator} is functorial.
\begin{pro}
Let $T$ and $T'$ be two nonabelian  embedding tensors on a Lie algebra $\g$ with respect to a coherent action  $(\h;\rho)$ and $(\phi_\g,\phi_\h)$   a homomorphism  from $T'$ to $T$. Then $\phi_\h$ is a homomorphism of \LL algebras from $(\h,[\cdot,\cdot]_\h,\rhd_{T'})$ to $(\h,[\cdot,\cdot]_\h,\rhd_{T})$. 
\end{pro}

At the end of this section, we show that a  \LL algebra
 $(\h,[-,-]_\h,\rhd)$ can  give rise to a nonabelian embedding tensor via two different ways such that the subadjacent Leibniz algebra is exactly the descendent Leibniz algebra of the nonabelian embedding tensor. First we make some preparations. Let $(\la,[-,-]_{\la})$ be a Leibniz algebra. Denote by $\Lei(\la)$ the ideal of squares generated by $[x,x]_\la$ for all $x\in\la$. Then
$$
\la_\Lie:=\la/\Lei(\la)
$$
is naturally a Lie algebra in which we denote the Lie bracket by $[-,-]_{\la_\Lie}$. The equivalence class of $x\in\la$ in $\la_\Lie$ will be denoted by $\bar{x}$. Moreover, $\Lei(\la)$ is contained in the left center of the Leibniz algebra $\la.$ We denote by
\begin{equation}
  \pr:\la\lon\la_\Lie
\end{equation}
the natural projection from the Leibniz algebra $\la$ to the Lie algebra $\la_\Lie$. Obviously, the projection $\pr$ preserves the bracket operation, i.e. the following equality holds:
\begin{equation}
  \pr[x,y]_\la=[\pr(x),\pr(y)]_{\la_\Lie},\quad \forall x,y\in \la.
\end{equation}

Let  $(\h,[-,-]_\h,\rhd)$ be a \LL algebra.  Since $[-,-]_\h$ is skewsymmetric, the ideal $\Lei(\h^\Lei)$ is generated by $x\rhd x$ for all $x\in\h$. Consider the quotient Lie algebra
$$
(\h^\Lei)_\Lie=\h^\Lei/\Lei(\h^\Lei).
$$
Define an action of $(\h^\Lei)_\Lie$ on the Lie algebra $(\h,[-,-]_\h)$ by
$$
\rho(\bar{u})v=u\rhd v,\quad \forall u,v\in\h.
$$
It is well defined. In fact, if $\bar{u}=\bar{u'}$, then $u-u'\in \Lei(\h^\Lei)$. By \eqref{Post-1}, we obtain $(u-u')\rhd v=0.$ Moreover, by \eqref{Post-2}, we have
$$
[\rho(\bar{u})v,w]_\h=[u\rhd v,w]_\h=0, \quad \rho(\bar{u})[v,w]_\h={u}\rhd[v,w]_\h=0.
$$

\begin{pro}
 Let $(\h,[-,-]_\h,\rhd)$  be a \LL algebra. Then the natural projection $\pr:\h\lon(\h^\Lei)_\Lie$ is a nonabelian embedding tensor on the Lie algebra $(\h^\Lei)_\Lie$ with respect to the coherent action $(\h;\rho)$ such that the induced descendent Leibniz algebra is exactly the subadjacent Leibniz algebra $\h^\Lei$.

\end{pro}
\begin{proof}
 It follows from
\begin{eqnarray*}
  [\pr(u),\pr(v)]_{(\h^\Lei)_\Lie}=\pr[u,v]=\pr([u,v]_\h+u\rhd v)=\pr([u,v]_\h+\rho(\pr({u}))v)
\end{eqnarray*}
directly.
\end{proof}

There is also another approach to construct a nonabelian embedding tensor from a \LL algebra. Let  $(\h,[-,-]_\h,\rhd)$ be a \LL algebra and $\Der(\h)$ the derivation Lie algebra of the Lie algebra $(\h,[-,-]_\h)$.
By  \eqref{Post-2}, $\huaL_u\in\overline{\Der}(\h)$, where $\huaL$ is given by \eqref{eq:L}.

\begin{pro}
 Let $(\h,[-,-]_\h,\rhd)$  be a \LL algebra. Then  $\huaL:\h\lon \overline{\Der}(\h)$ is a nonabelian embedding tensor on the Lie algebra $\overline{\Der}(\h)$ via the natural coherent action on the Lie algebra $(\h,[-,-]_\h)$ such that the induced descendent Leibniz algebra is exactly the subadjacent Leibniz algebra $\h^\Lei$.

\end{pro}
\begin{proof}
  By  \eqref{Post-1}, we have
\begin{eqnarray*}
  [\huaL_u,\huaL_v]=\huaL_{[u,v]_\h+\huaL_uv},\quad \forall u,v\in\h,
\end{eqnarray*}
which implies that $\huaL:\h\lon \overline{\Der}(\h)$ is a nonabelian embedding tensor. The other conclusion is obvious.
\end{proof}

\section{The controlling differential graded Lie algebra of nonabelian embedding tensors}\label{Deformation-DGLA}
In this section first we recall the controlling graded Lie algebra that characterize Leibniz algebras as Maurer-Cartan elements. Then we construct the controlling differential graded Lie algebra of  nonabelian embedding tensors.


A permutation $\sigma\in\mathbb S_n$ is called an {\bf $(i,n-i)$-shuffle} if $\sigma(1)<\cdots <\sigma(i)$ and $\sigma(i+1)<\cdots <\sigma(n)$. If $i=0$ or $n$, we assume $\sigma=\Id$. The set of all $(i,n-i)$-shuffles will be denoted by $\mathbb S_{(i,n-i)}$.

Let $\g$ be a vector space. We consider the graded vector space $$C^\bullet(\g,\g)=\oplus_{n=0}^{+\infty}C^n(\g,\g)=\oplus_{n=0}^{+\infty}\Hom(\otimes^{n+1}\g,\g).$$ It is known that $C^\bullet(\g,\g)$ equipped with the {\bf  Balavoine bracket} ~\cite{Bal}:
\begin{eqnarray}\label{leibniz-bracket}
[P,Q]_\B=P\bar{\circ}Q-(-1)^{pq}Q\bar{\circ}P,\,\,\,\,\forall P\in C^{p}(\g,\g),Q\in C^{q}(\g,\g),
\end{eqnarray}
is a graded Lie algebra,  where $P\bar{\circ}Q\in C^{p+q}(\g,\g)$ is defined by
$
P\bar{\circ}Q=\sum_{k=1}^{p+1}P\circ_k Q,
$
and $\circ_k$ is defined by
\begin{eqnarray*}
 \nonumber&&(P\circ_kQ)(x_1,\cdots,x_{p+q+1})\\
&=&\sum_{\sigma\in\mathbb S_{(k-1,q)}}(-1)^{(k-1)q}(-1)^{\sigma}P(x_{\sigma(1)},\cdots,x_{\sigma(k-1)},Q(x_{\sigma(k)},\cdots,x_{\sigma(k+q-1)},x_{k+q}),x_{k+q+1},\cdots,x_{p+q+1}).
\end{eqnarray*}

\emptycomment{
\begin{rmk}\label{NR-subalgebra}
Let $\g$ be a vector space. Then $\big(\oplus_{n=0}^{+\infty}\Hom(\wedge^{n+1}\g,\g),[\cdot,\cdot]_\NR\big)$ is a subalgebra of the graded Lie algebra $(\oplus_{n=0}^{+\infty}\Hom(\otimes^{n+1}\g,\g),[\cdot,\cdot]_\B)$, where  $[\cdot,\cdot]_\NR$ is the Nijenhuis-Richardson bracket \cite{NR,NR2}.
\end{rmk}
}


The following conclusion is straightforward.
\begin{lem}\label{lem:LeibnizMC}
For $\Omega\in C^{1}(\g,\g)=\Hom(\otimes^2\g,\g)$, we have
\begin{eqnarray*}
\half[\Omega,\Omega]_\B(x_1,x_2,x_3)=\Omega\bar{\circ}\Omega(x_1,x_2,x_3)=\Omega(\Omega(x_1,x_2),x_3)-\Omega(x_1,\Omega(x_2,x_3))
+\Omega(x_2,\Omega(x_1,x_3)).
\end{eqnarray*}
Thus, $\Omega$ defines a Leibniz algebra structure if and only if $[\Omega,\Omega]_\B=0$, i.e. $\Omega$ is a Maurer-Cartan element of the graded Lie algebra $( C^\bullet(\g,\g),[-,-]_\B)$.
\end{lem}

Let $(\g,[-,-]_\g)$ and $(\h,[-,-]_\h)$ be Lie algebras. Let $\rho:\g\lon\Der(\h)$ be a coherent action of $\g$ on $\h$.
Consider the graded vector space
$C^*(\h,\g):=\oplus_{n=1}^{+\infty}\Hom(\otimes^n\h,\g)$.
Define a linear map $\dM:\Hom(\otimes^n\h,\g)\to \Hom(\otimes^{n+1}\h,\g)$ by
\begin{eqnarray}
\nonumber&&(\dM f) ( v_1,\cdots, v_{n+1} )\\
\label{diff}&=&\sum_{1\le i< j\le n+1}(-1)^{n-1+i} f(v_1,\cdots,\hat{v}_i,\cdots,v_{j-1},[v_i,v_j]_{\h},v_{j+1},\cdots,v_{n+1}),
\end{eqnarray}
for all $f\in \Hom(\otimes^n\h,\g)$ and $v_1,\cdots, v_{n+1} \in\h$.
Also define a graded bracket operation
 $$\Courant{-,-}: \Hom(\otimes^{m}\h,\g)\times \Hom(\otimes^{n}\h,\g)\to \Hom(\otimes^{m+n}\h,\g),$$
{\footnotesize\begin{eqnarray}
&&\nonumber\Courant{\theta,\phi }(v_1,\cdots,v_{m+n})\\
\label{gla-embedding}&=&\sum_{k=1}^{m}\sum_{\sigma\in\mathbb S_{(k-1,n)}}(-1)^{(k-1)n+1}(-1)^{\sigma}\theta(v_{\sigma(1)},\cdots,v_{\sigma(k-1)},\rho(\phi (v_{\sigma(k)},\cdots,v_{\sigma(k+n-1)}))v_{k+n},v_{k+n+1},\cdots,v_{m+n})\\
&&\nonumber+\sum_{\sigma\in\mathbb S_{(m,n)}}(-1)^{mn+1}(-1)^{\sigma}[\theta(v_{\sigma(1)},\cdots,v_{\sigma(m)}),\phi (v_{\sigma(m+1)},\cdots,v_{\sigma(m+n-1)},v_{\sigma(m+n)})]_{\g}\\
&&\nonumber+\sum_{k=1}^{n}\sum_{\sigma\in\mathbb S_{(k-1,m)}}(-1)^{m(k+n-1)}(-1)^{\sigma}\phi (v_{\sigma(1)},\cdots,v_{\sigma(k-1)},\rho(\theta(v_{\sigma(k)},\cdots,v_{\sigma(k+m-1)}))v_{k+m},v_{k+m+1},\cdots,v_{m+n}),
\end{eqnarray}
}
for all  $\theta\in \Hom(\otimes^{m}\h,\g),~\phi \in \Hom(\otimes^{n}\h,\g)$ and $v_1,\cdots, v_{m+n} \in\h$.

\begin{thm}\label{dgla-deforamtion-homotopy}
Let $\rho:\g\lon\Der(\h)$ be a coherent action of a Lie algebra $(\g,[-,-]_\g)$ on a Lie algebra $(\h,[-,-]_\h)$. Then the triple $(C^*(\h,\g),\Courant{-,-},\dM)$ is a differential graded Lie algebra, whose Maurer-Cartan elements are  nonabelian embedding tensors.
\end{thm}
\begin{proof}
In short, the graded Lie algebra $(C^*(\h,\g),\Courant{-,-})$ is obtained via the derived bracket \cite{Kosmann-Schwarzbach-1}. In fact, the Balavoine bracket $[-,-]_{\B}$ associated to the direct sum vector space $\g\oplus \h$ gives rise to a graded Lie algebra $(\oplus_{k=0}^{+\infty}\Hom(\otimes^{k+1}(\g\oplus \h),\g\oplus \h),[-,-]_{\B})$. Obviously $\oplus_{k=0}^{+\infty}\Hom(\otimes^{k+1}\h,\g)$ is an abelian subalgebra. We denote the Lie brackets $[-,-]_\g$ and $[-,-]_\h$ by $\mu_\g$ and $\mu_\h$ respectively, and denote the hemisemidirect product Leibniz algebra structure on $\g\oplus\h$ by $\mu_\g\boxplus\rho$. Thus $\mu_\g\boxplus\rho$ and $\mu_\h$ are Maurer-Cartan elements of the graded Lie algebra $(\oplus_{k=0}^{+\infty}\Hom(\otimes^{k+1}(\g\oplus \h),\g\oplus \h),[-,-]_{\B})$. 
Define the derived bracket on the graded vector space $\oplus_{k=1}^{+\infty}\Hom(\otimes^k\h,\g)$ by
$$
  \Courant{\theta,\phi}:=(-1)^{m-1}[[\mu_\g\boxplus\rho,\theta]_{\B},\phi]_{\B},\quad\forall \theta\in\Hom(\otimes^m\h,\g),~\phi\in\Hom(\otimes^n\h,\g),
$$
which is exactly the bracket given by \eqref{gla-embedding}. By $[\mu_\g\boxplus\rho,\mu_\g\boxplus\rho]_{\B}=0,$ we deduce that $(C^*(\h,\g),\Courant{-,-})$ is a graded Lie algebra.

Define a linear map $\dM=:[\mu_\h,-]_{\B}$ on the graded space $\oplus_{k=1}^{+\infty}\Hom(\otimes^k(\g\oplus \h),\g\oplus \h)$. We obtain that $\dM$ is closed on the subspace $\oplus_{k=1}^{+\infty}\Hom(\otimes^k\h,\g)$, and is given by \eqref{diff}.

By $[\mu_\h,\mu_\h]_{\B}=0$, we obtain that $\dM\circ \dM=0.$ Moreover, by $\Img(\rho)\subset\Der(\h)$ and $[\rho(x)u,v]_\h=0$, we obtain that $[\mu_\g\boxplus\rho,\mu_\h]_{\B}=0.$ Then, we deduce that $\dM$ is a derivation of $(C^*(\h,\g),\Courant{-,-})$. Therefore,  $(C^*(\h,\g),\Courant{-,-},\dM)$ is a differential graded Lie algebra.

Finally, for a degree 1 element $T\in \Hom(\h,\g)$, we have
\begin{eqnarray*}
\big(\dM T+\frac{1}{2}\Courant{T,T}\big)(u,v)=-T[u,v]_\h-T(\rho(Tu)v)+[Tu,Tv]_\g.
\end{eqnarray*}
Thus, Maurer-Cartan elements are precisely nonabelian embedding tensors.
\end{proof}

Let $T:\h\to\g$ be a nonabelian embedding tensor. Since $T$ is a Maurer-Cartan element of the differential graded Lie algebra $(C^*(\h,\g),\Courant{-,-},\dM)$ by Theorem~\ref{dgla-deforamtion-homotopy}, we deduce that
$$d_{T}:=\dM+\Courant{T,-}$$
 is a graded derivation on the graded Lie
algebra $(C^*(\h,\g),\Courant{-,-})$ satisfying $d^2_{T}=0$.
  Therefore, $(C^*(\h,\g),\Courant{-,-},d_{{T}})$ is a differential graded Lie algebra.
This differential graded Lie algebra can control deformations of nonabelian embedding tensors. 

\begin{thm}\label{thm:deformation}
Let $T:\h\to\g$ be a nonabelian embedding tensor on a Lie algebra $\g$ with respect to a coherent action  $(\h;\rho)$ and $T':\h\to\g$ a
linear map. Then  $T+T'$ is still
a nonabelian embedding tensor on   $\g$ with respect to the coherent action  $(\h;\rho)$ if and only if $T'$ is a Maurer-Cartan
element of the differential graded Lie algebra
$(C^*(\h,\g),\Courant{-,-},d_{{T}})$.
\end{thm}
\begin{proof}
By Theorem \ref{dgla-deforamtion-homotopy}, $T+T'$ is
a nonabelian embedding tensor if and only if
\begin{eqnarray*}
  \dM (T+T')+\frac{1}{2}\Courant{T+T',T+T'}=0,
\end{eqnarray*}
which is equivalent to
$$
d_T  T'+\frac{1}{2}\Courant{T',T'}=0.
$$
That is, $T'$ is a Maurer-Cartan
element of
$(C^*(\h,\g),\Courant{-,-},d_{{T}})$.
\end{proof}

\section{Cohomologies of   nonabelian embedding tensors}\label{Cohomology}

In this section, we establish a cohomology theory for nonabelian embedding tensors via the Loday-Pirashvili   cohomology of the descendent Leibniz algebra given in Corollary \ref{average-Leibniz}.

\begin{defi}{\rm (\cite{Loday and Pirashvili})}  Let $(V;\rho^L,\rho^R)$ be a representation of a Leibniz algebra $(\la,[-,-]_{\la})$.
The {\bf  Loday-Pirashvili cohomology} of $\la$ with   coefficients in $V$ is the cohomology of the cochain complex $(C^\bullet(\la,V)=\oplus_{k=0}^{+\infty}C^k(\la,V),\partial)$, where $C^k(\la,V)=
\Hom(\otimes^k\la,V)$ and the coboundary operator
$\partial:C^k(\la,V)\to C^
{k+1}(\la,V)$
is defined by
\begin{eqnarray*}
(\partial f)(x_1,\cdots,x_{k+1})&=&\sum_{i=1}^{k}(-1)^{i+1}\rho^L(x_i)f(x_1,\cdots,\hat{x_i},\cdots,x_{k+1})+(-1)^{k+1}\rho^R(x_{k+1})f(x_1,\cdots,x_{k})\\
                      \nonumber&&+\sum_{1\le i<j\le k+1}(-1)^if(x_1,\cdots,\hat{x_i},\cdots,x_{j-1},[x_i,x_j]_\la,x_{j+1},\cdots,x_{k+1}),
\end{eqnarray*}
for all $x_1,\cdots, x_{k+1}\in\la$. 
\end{defi}

 Let $T:\h\to\g$ be a nonabelian  embedding tensor on a Lie algebra $\g$ with respect to a coherent action  $(\h;\rho)$.
By Corollary \ref{average-Leibniz}, $[u,v]_T:=[u,v]_\h+\rho(Tu)v$ defines a Leibniz algebra structure on $\h$. Furthermore, define $\rho^L:\h\lon\gl(\g)$ and $\rho^R:\h\lon\gl(\g)$ by
\begin{equation}
\rho^L(u)y:=[Tu,y]_\g,\quad \rho^R(v)x:=[x,Tv]_\g-T(\rho(x)v).
\end{equation}

\begin{lem}\label{new-rep}
  With above notations, $(\g;\rho^L,\rho^R)$ is  a representation of the Leibniz algebra $(\h,[-,-]_T)$.
\end{lem}
\begin{proof}
By \eqref{eq:nonET}, for all $u,v\in\h$ and $x\in\g$, we have
 \begin{eqnarray*}
  \Big(\rho^L([u,v]_T)-[\rho^L(u),\rho^L(v)]\Big)(x)&=&[T[u,v]_T,x]_\g-[Tu,[Tv,x]_\g]_\g+[Tv,[Tu,x]_\g]_\g\\
  &=&[[T u,Tv]_\g,x]_\g-[Tu,[Tv,x]_\g]_\g+[Tv,[Tu,x]_\g]_\g\\
  &=&0,
 \end{eqnarray*}
 which implies that $\rho^L([u,v]_T)=[\rho^L(u),\rho^L(v)]$.

 By \eqref{eq:extracon} and \eqref{eq:nonET}, for all $u,v\in\h$ and $x\in\g$, we have
 \begin{eqnarray*}
 && \Big( \rho^R([u,v]_{T})-[\rho^L(u),\rho^R(v)]\Big)(x)\\&=&[x,T[u,v]_{T}]_\g-T(\rho(x)[u,v]_{T})-[Tu,[x,Tv]_\g-T(\rho(x)v)]_\g+[[Tu,x]_\g,Tv]_\g-T(\rho([Tu,x]_\g)v)\\
 &=&[x,[Tu,{T}v]_\g]_\g-[Tu,[x,Tv]_\g]_\g+[[Tu,x]_\g,Tv]_\g\\
 &&-T[u,\rho(x)v]_\g-T\rho(x)\rho(Tu)v+[Tu,T(\rho(x)v)]_\g-T\rho(Tu)\rho(x)v+T\rho(x)\rho(Tu)v\\
 &=&0,
 \end{eqnarray*}
 which implies that $\rho^R([u,v]_T)=[\rho^L(u),\rho^R(v)]$.

 Finally, also by \eqref{eq:extracon} and \eqref{eq:nonET}, we have
 \begin{eqnarray*}
   \Big(\rho^R(v)\circ \rho^L(u)+\rho^R(v)\circ \rho^R(u)\Big)(x)&=& \rho^R(v)( -T\rho(x))=-[T\rho(x)u,Tv]_\g+T\rho(T\rho(x)u)v\\
   &=&-[\rho(x)u,v]_\g=0,
 \end{eqnarray*}
 which implies that $\rho^R(v)\circ \rho^L(u)+\rho^R(v)\circ \rho^R(u)=0.$ Therefore, $(\g;\rho^L,\rho^R)$ is  a representation of the Leibniz algebra $(\h,[-,-]_T)$.
\end{proof}

Let $\partial_T: \Hom(\otimes^k\h,\g)\to   \Hom(\otimes^{k+1}\h,\g)$ be the corresponding Loday-Pirashvili coboundary operator of the Leibniz algebra $(\h,[-,-]_T)$ with   coefficients in $(\g;\rho^L,\rho^R)$. More precisely, $\partial_T: \Hom(\otimes^k\h,\g)\to   \Hom(\otimes^{k+1}\h,\g)$ is given by
               \begin{eqnarray}
               \label{eq:defpt} && \partial_T \theta(u_1,\cdots,u_{k+1})\\
              \nonumber &=&\sum_{i=1}^{k}(-1)^{i+1}[Tu_i,\theta(u_1,\cdots,\hat{u}_i,\cdots, u_{k+1})]_\g\\
                 \nonumber&&+(-1)^{k+1}[\theta(u_1,\cdots,  u_{k}),Tu_{k+1}]_\g+(-1)^kT(\rho(\theta(u_1,\cdots, u_{k}))u_{k+1})\\
                \nonumber &&+\sum_{1\le i<j\le k+1}(-1)^{i}\theta(u_1,\cdots,\hat{u}_i,\cdots,u_{j-1},\rho(Tu_i)(u_j)+[u_i,u_j]_\h,u_{j+1},\cdots, u_{k+1}).
               \end{eqnarray}

Now we define a cohomology theory governing deformations of a nonabelian  embedding tensor $T:\h\lon\g.$ Define the space of  $0$-cochains $\frkC^0(T)$ to be $0$ and the space  of $1$-cochains $\frkC^1(T)$  to be $\g$. For $n\geq2$, define the space of $n$-cochains $\frkC^n(T)$ as $\frkC^n(T)=\Hom(\otimes^{n-1}\h,\g)$.

\begin{defi}\label{de:opcoh}
  Let $T:\h\lon \g$ be a nonabelian  embedding tensor on a Lie algebra $\g$ with respect to a coherent action  $(\h;\rho)$.    We define the {\bf cohomology of
  the nonabelian embedding tensor $T$} to be the cohomology of the cochain complex $( \frkC^\bullet(T)=\oplus _{k=0}^{+\infty} \frkC^k(T),\partial_T)$.
\end{defi}

Denote the set of $k$-cocycles by $\huaZ^k(T)$ and the set of $k$-coboundaries by $\huaB^k(T)$. Denote by \begin{equation}\huaH^k(T)=\huaZ^k(T)/\huaB^k(T), \quad k \geq 0, \label{eq:ocoh}\end{equation}
the corresponding $k$-th cohomology group.

Comparing the coboundary operators $\partial_T$ given above and the
operators $d_{{T}}=\dM+\Courant{T,-}$ defined by the Maurer-Cartan element $T$, we have

\begin{pro}\label{pro:danddT}
 Let ${T}:\h\to\g$ be a nonabelian  embedding tensor. Then we have
\begin{eqnarray}\label{differential-coho}
 \partial_T\theta=(-1)^{k-1}d_{T}\theta,\quad \forall \theta\in \Hom(\otimes^k\h,\g),\,\,k=1,2,\cdots,
 \end{eqnarray}
\end{pro}

\begin{proof}
For all $u_1,u_2,\cdots,u_{k+1}\in \h$ and $\theta\in \Hom(\otimes^k\h,\g)$, we have
\begin{eqnarray*}
&&(d_{T}\theta)(u_1,u_2,\cdots,u_{k+1})\\
&=&(\dM \theta+\Courant{T,\theta})(u_1,u_2,\cdots,u_{k+1})\\
                                &=&\sum_{1\le i< j\le k+1}(-1)^{k-1+i} \theta(u_1,\cdots,\hat{u}_i,\cdots,u_{j-1},[u_i,u_j]_{\h},u_{j+1},\cdots,u_{k+1})-T\big(\rho(\theta (u_{1},\cdots,u_{k}))u_{k+1}\big)\\
&&\nonumber+\sum_{i=1}^{k+1}(-1)^{k+1}(-1)^{i-1}[Tu_{i},\theta (u_1,\cdots,\hat{u}_i,\cdots, u_{k+1})]_{\g}\\
&&\nonumber+\sum_{1\le i< j\le k+1}(-1)^{(j-1+k-1)}(-1)^{j-i-1}\theta (u_1,\cdots,\hat{u}_i,\cdots,u_{j-1},\rho(T(u_{i}))u_{j},u_{j+11},\cdots,u_{k+1}),
\end{eqnarray*}
which implies that $\partial_T\theta=(-1)^{k-1}d_{T}\theta$. The proof is finished.
\end{proof}

\section{Linear deformations of    nonabelian  embedding tensors}\label{Linear-Deformation}
\label{sec:def}
In  this section, we study linear  deformations of a nonabelian  embedding tensor using the cohomology  theory  given in the
previous section. In particular, we introduce the notion of a
Nijenhuis element associated to a nonabelian  embedding tensor, which gives
rise to a trivial linear deformation of the
nonabelian  embedding tensor.  

    \begin{defi}
    Let $T$   be a nonabelian  embedding tensor  on a Lie algebra $\g$ with respect to a coherent action  $(\h;\rho)$ and $\frkT:\h\to\g$ a linear map. If  $T_t=T+t\frkT$ is still a nonabelian  embedding tensor on the Lie algebra $\g$ with respect to the coherent action  $(\h;\rho)$ for all $t\in\C$, we say that $\frkT$ generates a {\bf linear deformation} of the nonabelian  embedding tensor $T$.
    \end{defi}
By considering the coefficients of $t$ and $t^2$, one deduce that $T_t=T+t\frkT$ is a linear deformation of a nonabelian  embedding tensor $T$ if and only if
 for any $u,v\in \h$,
\begin{eqnarray}
~[Tu,\frkT v]_\g+[\frkT u,Tv]_\g&=&T(\rho(\frkT u)v)+\frkT(\rho(Tu)v+[u,v]_\h),\mlabel{eq:deform1}\\
~[\frkT u,\frkT v]_\g&=&\frkT(\rho(\frkT u)v).
\mlabel{eq:deform2}
\end{eqnarray}
Note that Eq.~\eqref{eq:deform1} means that $\frkT$ is a 1-cocycle of the Leibniz algebra $(\h,[\cdot,\cdot]_T)$ with coefficients in the representation $(\g;\rho^L,\rho^R)$   given by Lemma \ref{new-rep} and Eq.~\eqref{eq:deform2} means that $\frkT$ is an embedding tensor on the Lie algebra $\g$ with respect to the representation $(\h;\rho)$.

 Let $(\la,[-,-])$ be a Leibniz algebra and $\omega:\otimes^2\la\to \la$ be a linear map. If for any $t\in\C$, the multiplication $[-,-]_t$ defined by
$$
[u,v]_t=[u,v]+t\omega(u,v), \;\forall u,v\in \la,
$$
also gives a Leibniz algebra structure, we say that $\omega$ generates a {\bf linear deformation} of the Leibniz algebra $(\la,[-,-])$.

The two types of linear deformations are related as follows.

\begin{pro}
 If $\frkT$ generates a linear deformation of a nonabelian  embedding tensor  $T$ on a Lie algebra $\g$ with respect to a coherent action  $(\h;\rho)$, then the product $\omega_\frkT$ on $\h$ defined by
   $$
   \omega_\frkT(u,v)=\rho(\frkT u)v,\quad\forall u,v\in \h,
   $$
generates a linear deformation of the associated Leibniz algebra $(\h,[-,-]_{T})$.
\end{pro}

\begin{proof}
We denote by $[-,-]_t$ the corresponding Leibniz algebra structure associated to the nonabelian  embedding tensor $T+t\frkT$. Then we have
\begin{eqnarray*}
[u,v]_t&=&\rho((T+t\frkT)u)v+[u,v]_\h\\
&=&\rho(Tu)v+[u,v]_\h+t\rho(\frkT
u)v\\
&=&[u,v]_T+t\omega_\frkT (u,v), \quad \forall u,v\in V,
\end{eqnarray*}
which implies that $\omega_\frkT$ generates a  linear deformation of $(\h,[-,-]_{T})$.\end{proof}

\begin{defi} Let $T$ be a nonabelian  embedding tensor on a Lie algebra $\g$ with respect to a coherent action  $(\h;\rho)$. Two
linear deformations $T^1_t=T+t\frkT_1$ and
$T^2_t=T+t\frkT_2$ are said to be {\bf equivalent} if there exists
an $x\in\g$ such that $({\Id}_\g+t\ad_x,{\Id}_\h+t\rho(x))$ is a
homomorphism   from $T^2_t$ to $T^1_t$. In particular, a
 linear deformation $T_t=T+t\frkT$ of a
nonabelian  embedding tensor $T$ is said to be {\bf trivial} if there exists
an $x\in\g$ such that $({\Id}_\g+t\ad_x,{\Id}_\h+t\rho(x))$ is a
homomorphism   from $T_t$ to $T$.
\end{defi}

Let $({\Id}_\g+t\ad_x,{\Id}_\h+t\rho(x))$ be a homomorphism from
$T^2_t$ to $T^1_t$. Then ${\Id}_\g+t\ad_x$ and ${\Id}_\h+t\rho(x)$ are  Lie algebra
endomorphisms of $(\g,[-,-]_\g)$ and $(\h,[-,-]_\h)$ respectively. Thus, we have
\begin{eqnarray*}
({\Id}_\g+t\ad_x)[y,z]_\g&=&[({\Id}_\g+t\ad_x)(y),({\Id}_\g+t\ad_x)(z)]_\g, \;\forall y,z\in \g,\\
({\Id}_\h+t\rho(x))[u,v]_\h&=&[({\Id}_\h+t\rho(x))(u),({\Id}_\h+t\rho(x))(v)]_\h, \;\forall u,v\in\h,
\end{eqnarray*}
  which implies that $x$ satisfies
\begin{equation}
 [[x,y]_\g,[x,z]_\g]_\g=0,\quad \forall y,z\in\g.
 \mlabel{eq:Nij1}
 \end{equation}
Then by Eq.~\eqref{defi:isocon1}, we get
$$
(T+t\frkT_1)({\Id}_\h+t\rho(x))(u)=({\Id}_\g+t\ad_x)(T+t\frkT_2)(u),\quad\forall u\in \h,
$$
which implies
\begin{eqnarray}
 (\frkT_2-\frkT_1)(u)&=&T\rho(x)u-[x,Tu]_\g,\mlabel{eq:deforiso1} \\
  \frkT_1\rho(x)(u)&=&[x,\frkT_2u]_\g, \; \forall u\in \h.
  \mlabel{eq:deforiso2}
\end{eqnarray}
Finally, Eq.~\eqref{defi:isocon3} gives
$$
({\Id}_\h+t\rho(x))\rho(y)(u)=\rho(({\Id}_\g+t\ad_x)(y))({\Id}_\h+t\rho(x))(u),\quad \forall y\in\g, u\in \h,
$$
which implies that $x$ satisfies
\begin{equation}
  \rho([x,y]_\g)\rho(x)=0,\quad\forall y\in\g.\mlabel{eq:Nij3}
\end{equation}
Note that Eq.~\eqref{eq:deforiso1} means that $\frkT_2-\frkT_1=\partial_T x$. Thus, we have the following result:

\begin{thm}\label{thm:iso3} Let $T$ be a  nonabelian  embedding tensor  on a Lie algebra $\g$ with respect to a coherent action  $(\h;\rho)$.
  If two linear deformations $T^1_t=T+t\frkT_1$ and $T^2_t=T+t\frkT_2$ are equivalent, then $\frkT_1$ and $\frkT_2$ are in the same cohomology class of $ \huaH^2(T)$.
\end{thm}

    \begin{defi}
Let $T$ be a   nonabelian  embedding tensor on a Lie algebra $\g$ with respect to a coherent action  $(\h;\rho)$. An element $x\in\g$ is called a {\bf Nijenhuis element} associated to $T$ if $x$ satisfies Eqs.~\eqref{eq:Nij1}, \eqref{eq:Nij3} and the equation
      \begin{eqnarray}
        ~[x,T\rho(x)u-[x,Tu]_\g]_\g=0,\quad \forall u\in \h.
         \mlabel{eq:Nijenhuis}
        \end{eqnarray}
   Denote by $\Nij(T)$ the set of Nijenhuis elements associated to a nonabelian  embedding tensor $T$.
    \end{defi}

By Eqs.~\eqref{eq:Nij1}-\eqref{eq:Nij3}, it is obvious that a
trivial linear deformation gives rise to a
Nijenhuis element. Conversely, a Nijenhuis element can also generate a trivial linear deformation as the following theorem shows.

  \begin{thm}\label{thm:trivial}
   Let $T$ be a nonabelian  embedding tensor on a Lie algebra $\g$ with respect to a coherent action  $(\h;\rho)$. Then for any  $x\in \Nij(T)$, $T_t=T+t \frkT$ with $\frkT=\partial_T x$ is a trivial linear  deformation of the nonabelian  embedding tensor  $T$.
\end{thm}
We need the following lemma to prove this theorem.

\begin{lem}\label{lem:isomorphism}
Let $T$ be a nonabelian  embedding tensor on a Lie algebra $\g$ with respect to a coherent action  $(\h;\rho)$.  Let $\phi_\g:\g\to\g$ and $\phi_\h:\h\to \h$ be two Lie algebra isomorphisms  such that Eq.~\eqref{defi:isocon1} and Eq.~\eqref{defi:isocon3} hold. Then $\phi_\g^{-1}\circ{T}\circ\phi_\h$ is a nonabelian  embedding tensor on the Lie algebra $\g$ with respect to the coherent action  $(\h;\rho)$.
\end{lem}

\begin{proof}
 It follows from
straightforward computations.
\end{proof}

{\bf The proof of Theorem \ref{thm:trivial}:} 
For any Nijenhuis element $x\in\Nij({T})$, we define a linear map $\frkT:\h\to \g$ by
 \begin{eqnarray}\label{trivial-genetator}
\frkT(u)=(\partial_T x)(u)=T\rho(x)u-[x,Tu]_\g,~\forall u\in\h.
 \end{eqnarray}
Let $T_t=T+t \frkT$. By the definition of Nijenhuis elements, for all $y,z\in\g, u,v\in \h$,  we have
\begin{eqnarray*}
      ({\Id}_\g+t\ad_x)[y,z]_\g&\stackrel{\eqref{eq:Nij1}}{=}&[({\Id}_\g+t\ad_x)(y),({\Id}_\g+t\ad_x)(z)]_\g,\\
      ({\Id}_\h+t\rho(x))[u,v]_\h&\stackrel{\eqref{eq:extracon}}{=}&[({\Id}_\h+t\rho(x))(u),({\Id}_\h+t\rho(x))(v)]_\h,\\
      ({\Id}_\g+t\ad_x)\circ T_t&\stackrel{\eqref{eq:Nijenhuis}}{=}&T\circ({\Id}_\h+t\rho(x)),\\
        ({\Id}_\h+t\rho(x))\rho(y)u&\stackrel{\eqref{eq:Nij3}}{=}&\rho( ({\Id}_\g+t\ad_x)y)({\Id}_\h+t\rho(x))(u).
\end{eqnarray*}
Since $\ad_x$ and $\rho(x)$ are linear transformations of finite-dimensional $\C$-vector spaces, it follows that  ${\Id}_\g+t\ad_x$ and ${\Id}_\h+t\rho(x)$ are Lie algebra isomorphisms for $|t|$ sufficiently small. Thus, by Lemma \ref{lem:isomorphism},
\begin{eqnarray*}
&&({\Id}_\g+t\ad_x)^{-1}\circ{T}\circ ({\Id}_\h+t\rho(x))\\
   &=&\sum_{i=0}^{+\infty}(-t\ad_x)^i\circ(T+tT\rho(x))\\
   &=&T+t(T\rho(x)-\ad_x\circ T)+\sum_{i=1}^{+\infty}(-1)^it^{i+1}\ad_{x}^i\circ(T\rho(x)-\ad_x\circ T)\\
   &\stackrel{\eqref{eq:Nijenhuis}}{=}&T+t(T\rho(x)-\ad_x\circ T)=T+t \frkT=T_t
\end{eqnarray*}
 is a nonabelian  embedding tensor on the Lie algebra $(\g,[-,-]_\g)$ with respect to the coherent  action $(\h;\rho)$ for $|t|$ sufficiently small.
  Thus, $\frkT$
  given by
Eq.~\eqref{trivial-genetator} satisfies the conditions
\eqref{eq:deform1} and \eqref{eq:deform2}. Therefore,
${T}_t$ is a nonabelian  embedding tensor for all
$t$, which means that $\frkT$
given by
Eq.~\eqref{trivial-genetator} generates a linear deformation of $T$. It is straightforward to see  that
this linear deformation is trivial.\qed\vspace{2mm}

Now we recall the notion of a Nijenhuis operator on a Leibniz algebra, which gives rise to a trivial linear deformation of a Leibniz algebra.

\begin{defi}{\rm (\cite{CGM})}
  A linear map $N:\la\to \la$ on a Leibniz algebra $(\la,[-,-]_\la) $ is called a {\bf Nijenhuis operator} if
  \begin{equation}
   [Nu,Nv]_\la=N([Nu,v]_\la+[u,Nv]_\la-N[u,v]_\la),\quad u,v\in\la.
  \end{equation}
\end{defi}

For its connection with a Nijenhuis element associated to
a nonabelian  embedding tensor, we have the following result.

\begin{pro}
Let $x\in\g$ be a Nijenhuis element associated to a
nonabelian  embedding tensor $T$ on a Lie algebra $\g$ with respect to a coherent action  $(\h;\rho)$. Then $\rho(x)$ is a Nijenhuis operator
on the descendent Leibniz algebra $(\h,[-,-]_{T})$.
\end{pro}

\begin{proof}
For all $u,v\in \h$, we have
    \begin{eqnarray*}
    &&\rho(x)([\rho(x)u,v]_T+[u,\rho(x)v]_T-\rho(x)[u,v]_T)-[\rho(x)u,\rho(x)v]_T\\
&\stackrel{\eqref{new-leibniz}}{=}&\rho(x)\Big(\rho(T\rho(x)u)v+[\rho(x)u,v]_\h+\rho(Tu)\rho(x)v+[u,\rho(x)v]_\h-\rho(x)(\rho(Tu)v+[u,v]_\h)\Big)\\
&&-\rho(T\rho(x)u)\rho(x)v-[\rho(x)u,\rho(x)v]_\h\\
&\stackrel{\eqref{eq:extracon}}{=}&\rho(x)\big(\rho(T\rho(x)u)v+\rho(Tu)\rho(x)v-\rho(x)\rho(Tu)v\big)-\rho(T\rho(x)u)\rho(x)v\\
&=&\rho([x,T\rho(x)u]_\g)v+\rho(x)\rho([Tu,x]_\g)v\\
&\stackrel{\eqref{eq:Nijenhuis}}{=}&\rho([x,[x,Tu]_\g]_\g)v+\rho(x)\rho([Tu,x]_\g)v{=}-\rho([x,Tu]_\g)\rho(x)v\\
&\stackrel{\eqref{eq:Nij3}}{=}&0.
    \end{eqnarray*}
Thus, we deduce that $\rho(x)$ is a Nijenhuis operator
on the Leibniz algebra $(\h,[-,-]_{T})$.
\end{proof}

We end this  section by giving a concrete example of a linear deformation  of a nonabelian  embedding tensor $T$ on the Heisenberg Lie algebra $H_3(\mathbb C)$.

\begin{ex}{\rm
Let $H_3(\mathbb C)$ be the Heisenberg Lie algebra. Then $T_t=\left(\begin{array}{ccc}\frac{t+\sqrt{t^2+4t}}{2}&0&0\\
0&\frac{t+\sqrt{t^2+4t}}{2}&0\\
a&b&t\end{array}\right)$
is a linear deformation  of the  nonabelian  embedding tensor $T=\left(\begin{array}{ccc}0&0&0\\
0&0&0\\
a&b&0\end{array}\right)$ on $H_3(\mathbb C)$ with respect to the adjoint action for all $a,b\in\mathbb C$.
}
\end{ex}

\noindent
{\bf A conflict of interest statement.}
On behalf of all authors, the corresponding author states that there is no conflict of interest.

\noindent
{\bf Data availability.} Data sharing not applicable to this article as no datasets were generated or analysed during the current study.

\vspace{2mm}
\noindent
{\bf Acknowledgements.} This research is supported by NSFC (11922110,12001228). We give warmest thank to Jim Stasheff, Thomas Strobl and Chenchang Zhu for helpful comments.

 \end{document}